\documentclass{amsart}
\usepackage{latexsym}
\usepackage{amsthm}
\usepackage{amsmath}
\usepackage{amsfonts}
\usepackage{amssymb}
\usepackage[dvips]{graphicx}
\usepackage[usenames,dvipsnames]{color}

\input xy
\xyoption{all}
\newdir{ >}{{}*!/-7pt/\dir{>}} 
\SelectTips{cm}{}

\theoremstyle{plain}

\newtheorem{theo}{Theorem}[section]
\newtheorem{lemma}[theo]{Lemma}
\newtheorem{propo}[theo]{Proposition}
\newtheorem{coro}[theo]{Corollary}

\theoremstyle{definition}
\newtheorem{defi}[theo]{Definition}

\theoremstyle{remark}
\newtheorem{rem}[theo]{Remark}

\newcommand\colim{\operatorname{colim}}
\newcommand\id{\operatorname{id}}

\newcommand\sets{\mathcal{S}\mathrm{ets}}
\newcommand\ssets{\mathrm{s}\mathcal{S}\mathrm{ets}}

\newcommand\sOper{\mathrm{s}\mathcal{O}\mathrm{per}}

\newcommand\sCat{\mathrm{s}\mathcal{C}\mathrm{at}}
\newcommand\Cat{\mathcal{C}\mathrm{at}}
\newcommand\ECat[1]{{#1}\text{-}\Cat}
\newcommand\Oper{\mathcal{O}\mathrm{per}}
\newcommand\Alg{\mathcal{A}\mathrm{lg}}
\newcommand\EAlg[1]{{#1}\text{-}\Alg}

\newcommand{\myrightleftarrows}[1]{\mathrel{\substack{\xrightarrow{#1} \\[-.8ex] \xleftarrow{#1}}}}

\newcommand{\lmorp}[3]{#1 \colon #2 \longrightarrow #3}
\newcommand{\morp}[3]{#1 \colon #2 \to #3}

\newcommand{\gfun}[5]{
    \begin{array}{r@{}c@{\hspace{0.1cm}}c@{\hspace{0.1cm}}l}
        #1 \colon & #2 & \longrightarrow & #3 \\
        & #4 & \longmapsto & #5\\
    \end{array}
}

\newcommand{\sAt}{\mathrm{s}\mathcal{AT}\mathrm{h}}
\newcommand{\At}{\mathcal{AT}\mathrm{h}}
\newcommand{\EAt}[1]{{#1}\text{-}\At}
\newcommand{\sAlg}[1]{\mathrm{s}\mathcal{A}\mathrm{lg}(#1)}
\newcommand{\at}{\mathcal}
\newcommand{\wm}{\mathrm{wm}}
\newcommand{\cat}{\mathcal}

\newcommand{\str}{\mathbf}
\newcommand{\Str}[1]{\mathsf{Sq}(#1)}
\newcommand{\crep}[1]{h_{#1}}
\newcommand{\Na}{\mathsf{N}}
\newcommand{\Sa}{\mathsf{S}}
\newcommand{\op}{^{\mathrm{op}}}
\newcommand{\bcn}[4]{\mathbf{B}_{#4}(#1, #2, #3)}
\newcommand{\bc}[3]{\bcn{#1}{#2}{#3}{\bullet}}
\newcommand{\dbcn}[4]{{B}_{#4}(#1, #2, #3)}
\newcommand{\dbc}[3]{\dbcn{#1}{#2}{#3}{\bullet}}
\newcommand{\bcf}[2]{\overline{#1(#2)}}

\newcommand{\nervpr}[2]{\overline{#1}(#2)}
\newcommand{\N}{\mathbb{N}}
\newcommand{\Fin}{\mathrm{Fin}}
\newcommand{\Sym}{\mathbf{\Sigma}}
\newcommand{\Sign}[1]{{\mathrm{Sign}(#1)}}
\newcommand{\Fign}[1]{{\mathrm{CSign}(#1)}}
\newcommand{\Ord}[1]{\mathrm{Ord}_{#1}}

\newcommand{\Clo}{\mathrm{COper}}

\newcommand{\oat}[1]{\mathbb{T}(#1)}
\newcommand{\oatf}{\mathbb{T}}
\newcommand{\ope}{\mathcal}

\newcommand{\Dayp}{\overline{\otimes}}

\newcommand{\cAtf}{\mathbb{P}}
\newcommand{\cAt}[1]{\cAtf{#1}}
\newcommand{\Mcc}{\cat{M}}

\newcommand{\EOper}[1]{{#1}\text{-}\Oper}
\newcommand{\EClo}[1]{{#1}\text{-}\Clo}
\newcommand{\EOperc}[2]{\EOper{#1}_{#2}}
\newcommand{\ECloc}[2]{\EClo{#1}_{#2}}
\newcommand{\EAtc}[2]{\EAt{#1}_{#2}}

\numberwithin{equation}{section}

\addtocontents{toc}{\protect\setcounter{tocdepth}{1}}

\begin{document}
\title[On Morita equivalences of simplicial algebraic theories and operads]{On Morita weak equivalences of simplicial algebraic theories and operads}

\author[G. Caviglia]{Giovanni Caviglia}
\address{Radboud Universiteit Nijmegen, Institute for
Mathematics, Astrophysics, and Particle Physics, Heyendaalseweg 135, 6525 AJ
Nijmegen, The Netherlands}
\email{giova.caviglia@gmail.com}

\author[J.\,J. Guti\'errez]{Javier J.~Guti\'errez}
\address{Departament de Matem\`atiques i Inform\`atica,
Universitat de Barcelona (UB),
Gran Via de les Corts Catalanes 585,
08007 Barcelona, Spain}
\email{javier.gutierrez@ub.edu}
\urladdr{http://www.ub.edu/topologia/gutierrez}

\begin{abstract}
Generalizing a classical result of Dwyer and Kan for simplicial categories, we characterize the morphisms of multi-sorted simplicial algebraic theories and simplicial  coloured operads which induce a Quillen equivalence between the corresponding categories of algebras. 
\end{abstract}

\maketitle

\section{Introduction}\label{sect:intro}
Given a small category $\mathcal{C}$, we can think of the category of functors $\sets^\mathcal{C}$ as the category of \emph{representations} (in $\sets$) of $\mathcal{C}$. Two categories need not to be equivalent to have equivalent categories of representations. In fact, recall that  every functor $f\colon\mathcal{C}\to \mathcal{D}$ between small categories induces an adjunction between the corresponding categories of representations
\[
    f_!:\sets^\mathcal{C}\myrightleftarrows{\rule{0.4cm}{0cm}}\sets^\mathcal{D}: f^*,
\]
where the right adjoint $f^*$ is precomposition with $f$. The following is a classical result in category theory (see for instance~\cite{EZ76}, \cite{BD86}).
\begin{theo}\label{thm:moritaCat}
For every functor $f\colon \mathcal{C} \to \mathcal{D}$, the adjunction $(f_!, f^*)$ is an equivalence of categories if and only if $f$ is fully faithful and essentially surjective up to retracts.    
\end{theo}

A functor satisfying the hypothesis of Theorem~\ref{thm:moritaCat} is called a \emph{Morita equivalence}.
In \cite{DK87}, Dwyer and Kan extended this characterization to the homotopical setting. Let $\ssets$ be the category of simplicial sets equipped with the Kan--Quillen model structure and let $\mathcal{C}$ be a small simplicial category, that is, a category enriched in simplicial sets. The category of simplicial functors $\ssets^{\mathcal{C}}$ can be endowed with the \emph{projective model structure}, in which the weak equivalences and fibrations are defined levelwise. The projective model structure models the \emph{homotopy representations} of $\mathcal{C}$.

To every simplicial category~$\mathcal{C}$ we can associate (functorially) a category $\pi_0(\mathcal{C})$, called the \emph{path component category} of $\mathcal{C}$, which has the same objects as $\mathcal{C}$, and whose set of morphisms from $x$ to $y$ in $\pi_0(\mathcal{C})$ is $\pi_0(\mathcal{C}(x,y))$.

We call a functor between small simplicial categories $f\colon \mathcal{C} \to \mathcal{D}$ a \emph{Morita weak equivalence} if it is \emph{homotopically fully faithful} (that is, $f\colon \mathcal{C}(x, y) \to \mathcal{D}(f(x), f(y))$ is a weak equivalence for every $x,y$ in $\mathcal{C}$) and \emph{homotopically essentially surjective up to retracts} (that is, if $\pi_0(f)$ is essentially surjective up to retracts).

The result of Dwyer and Kan \cite[Theorem 2.1]{DK87} can be stated as follows (even though it was originally formulated not making use of the language of model categories).
\begin{theo}[Dwyer--Kan]\label{thm:DK}
Let $f\colon \mathcal{C}\to \mathcal{D}$ be a functor between small simplicial categories. The following are equivalent:
    \begin{itemize}
    \item[{\rm (i)}] The induced Quillen adjunction between the projective model structures
    \[
    f_!:\ssets^{\mathcal{C}}\myrightleftarrows{\rule{0.4cm}{0cm}} \ssets^{\mathcal{D}}:f^*
    \]
    is a Quillen equivalence.
    \item[{\rm (ii)}] The functor $f$ is homotopically fully faithful and homotopically essentially surjective up to retracts.
    \end{itemize}
\end{theo}

The goal of this paper is to extend Theorem~\ref{thm:moritaCat} and Theorem~\ref{thm:DK} to the case of (simplicial) multi-sorted algebraic theories and (simplicial) coloured operads.

Algebraic theories and coloured operads can be regarded as extensions of the concept of category that present algebraic structures with operations with multiple inputs and one output in cartesian and symmetric monoidal categories, respectively, in contrast with categories that can present structures with operations with one input and one output only (but, consequently, admit representations in arbitrary categories that do not need to be cartesian or symmetric monoidal).

Representations of algebraic theories and operads are called \emph{algebras}. The \emph{category of algebras} over an
algebraic theory $\mathcal{T}$ or operad $\mathcal{O}$ will be denoted by $\Alg(\mathcal{T})$ or $\Alg(\mathcal{O})$, respectively. Every morphism of algebraic theories $f\colon \mathcal{S} \to \mathcal{T}$ or operads $f\colon \mathcal{O} \to \mathcal{P}$ induces an adjunction:
\[
f_!:\Alg(\mathcal{S}) \myrightleftarrows{\rule{0.4cm}{0cm}}\Alg(\mathcal{T}): f^* \quad\mbox{or}\quad f_!:\Alg(\mathcal{O}) \myrightleftarrows{\rule{0.4cm}{0cm}}\Alg(\mathcal{P}): f^* ,
\]
respectively. Our generalizations of Theorem \ref{thm:moritaCat} are Theorem~\ref{theo:moritaAtSet} and Theorem~\ref{theo:moritaOpSet}, and can be subsumed as follows: 
\begin{theo}
    Let $f$ be a morphism of multi-sorted algebraic theories or coloured operads in $\sets$. The following are equivalent:
    \begin{itemize}
        \item[{\rm (i)}] The induced adjunction $(f_!, f^*)$ is an equivalence of categories.
        \item[{\rm (ii)}] The morphism $f$ is fully faithful and essentially surjective up to retracts.
    \end{itemize}
\end{theo}
The precise definition of fully faithful and essentially surjective up to retracts in each case will be given in the corresponding sections of the paper.

In the homotopical setting, we consider simplicial multi-sorted algebraic theories and simplicial operads. For every simplicial algebraic theory $\mathcal{T}$ or simplicial operad $\mathcal{O}$, its category of algebras in $\ssets$, that we denote by $\sAlg{\mathcal{T}}$ or $\sAlg{\mathcal{O}}$, respectively, admit the projective model structure, which models the homotopy theory of the respective algebras.

As in the case of simplicial categories, every morphism of simplicial algebraic theories $f\colon \mathcal{S} \to \mathcal{T}$ or simplicial operads $f\colon \mathcal{O} \to \mathcal{P}$ induces a Quillen adjunction between the projective model structures:
\[
f_!:\sAlg{\mathcal{S}} \myrightleftarrows{\rule{0.4cm}{0cm}}\sAlg{\mathcal{T}}: f^* \quad\mbox{or}\quad f_!:\sAlg{\mathcal{O}} \myrightleftarrows{\rule{0.4cm}{0cm}}\sAlg{\mathcal{P}}: f^* ,\]
respectively. The generalization of Theorem \ref{thm:DK} for algebraic theories, presented in Section \ref{sect:ch_colour} as Corollary \ref{cor:char_Morita_at} can be stated as follows:
\begin{theo}
 Let $f$ be a morphism of simplicial  multi-sorted algebraic theories. The following are equivalent:
    \begin{itemize}
        \item[{\rm (i)}] The adjunction $(f_!, f^*)$ is a Quillen equivalence between the projective model structures on the corresponding categories of simplicial algebras.
        \item[{\rm (ii)}] The morphism $f$, seen as a functor between simplicial categories, is a Morita weak equivalence.
    \end{itemize}
\end{theo} 
The generalization of Theorem \ref{thm:DK} that we obtain in the operadic case is the following:
\begin{theo}\label{thm:moritasOp}
Let $f$ be a morphism between $\Sigma$-cofibrant simplicial coloured operads. The following are equivalent:
    \begin{itemize}
        \item[{\rm (i)}]The adjunction $(f_!, f^*)$ is a Quillen equivalence between the projective model structures on the corresponding categories of simplicial algebras.
        \item[{\rm (ii)}] The morphism $f$ is homotopically fully faithful and homotopically essentially surjective up to retracts.
    \end{itemize}
\end{theo} 
The morphisms between simplicial operads that are homotopically fully faithful and homotopically essentially surjective up to retracts are also called \emph{Morita weak equivalences}.

The Cisinski--Moerdijk model structure on simplicial operads models the homotopy theory of $\infty$-operads. However, the homotopy theory of $\mathcal{O}$-algebras in the homotopy category of $\ssets$ for a simplicial operad $\mathcal{O}$, seen as an $\infty$-operad, is not equivalent to the homotopy theory of $\sAlg{\mathcal{O}}$ (with the projective model structure), but it is rather represented by $\sAlg{\widehat{\mathcal{O}}}$, where $\widehat{\mathcal{O}}$ is a cofibrant resolution of~$\mathcal{O}$.

Cofibrant simplicial operads represent homotopy invariant algebraic structures. Since every cofibrant operad is, in particular, $\Sigma$-cofibrant, Theorem \ref{thm:moritasOp} shows that, restricted to cofibrant operads, the class of Morita weak equivalences is precisely the class of morphisms which induce homotopical equivalences between the homotopy invariant algebraic structures presented. This implies that the Morita model structure for simplicial operads presented in \cite[Theorem 5.7]{CG19} can be rightly called the \emph{homotopy theory of homotopy invariant algebraic structures}. This conclusion was one of the main motivations to write this paper.

We would like to stress that our proof of Theorem~\ref{thm:moritasOp} is highly dependent on the corresponding result for algebraic theories, that is, Corollary~\ref{cor:char_Morita_at}. Moreover, Theorem~\ref{thm:moritasOp} provides a nice characterization of the weak equivalences of the Morita model structure for simplicial operads from \cite[Theorem 5.7]{CG19}.

\bigskip
\noindent {\bf Organization of the paper.}
We start with the case of multi-sorted algebraic theories in Section~\ref{sect:Morita_AT}. The first part of this section contains the results for algebraic theories in $\sets$ while the second part is concerned with the homotopical case.
We then start the transition from algebraic theories to operads. Section~\ref{sect:Oper_AT} recalls the strong link between the concept of operad and algebraic theory, explained by Kelly in \cite{Ke05}, that will permit us to recast the desired characterization of Morita equivalences between coloured operads to the corresponding characterization for algebraic theories.
Finally Section~\ref{sec:operads} contains the main results about coloured operads. We deduce the characterization of Morita equivalences for operads in sets and we conclude by proving the homotopical analogue for simplicial operads. 

\bigskip
\noindent {\bf Acknowledgements.} We would like to thank the referee for very useful comments and suggestions concerning the contents of the paper. The first named author was supported by the Spanish Ministry of Economy under the grants MTM2016-76453-C2-2-P (AEI/FEDER, UE) and RYC-2014-15328 (Ram\'on y Cajal Program).

\section{Morita equivalences of multi-sorted algebraic theories}\label{sect:Morita_AT}

\subsection{Multi-sorted algebraic theories}
We start by recalling the definition of multi-sorted algebraic theory. To encompass both the non-enriched and the simplicial case, we give the definition of multi-sorted algebraic theory enriched in an arbitrary cocomplete cartesian category $\mathcal{M}$. However, in the paper we will only consider the cases in which $\mathcal{M}$ is the category of sets or the category of simplicial sets.

Let $S$ be a set and let $\Str{S}$ denote the set of finite words or finite sequences over~$S$. Let $\Mcc$ be a cocomplete cartesian category. An \emph{$S$-sorted $\Mcc$-algebraic theory} $\mathcal{T}$ is a small category enriched in $\Mcc$ with finite products whose objects consist of finite sequences $\langle a_1,\ldots, a_n\rangle $ in $\Str{S}$ with $n\ge0$, and such that for every
$\str{a}=\langle a_1,\ldots, a_n\rangle$ and $\str{b}=\langle b_1,\ldots, b_k\rangle$ in $\Str{S}$ there is an isomorphism
$$
\str{a}\str{b}=\langle a_1,\ldots, a_n,b_1,\dots, b_k\rangle\cong \langle a_1,\ldots, a_n\rangle\times\langle b_1,\ldots, b_k\rangle=\str{a}\times \str{b}.
$$
A \emph{multi-sorted $\Mcc$-algebraic theory} is an $S$\nobreakdash-sorted $\Mcc$-algebraic theory for some set~$S$.

A morphism from an $S$-sorted $\Mcc$-algebraic theory $\mathcal{T}$ to an $S'$-sorted $\Mcc$-algebraic theory $\mathcal{T}'$ is a pair $(f, \varphi)$ where $\varphi\colon S\to S'$ is a function of sets and $f\colon \mathcal{T}\to\mathcal{T}'$ is a product-preserving $\Mcc$-functor such that $f(a)=\varphi(a)$ for every $a$ in $S$. We will denote by $\EAt{\Mcc}$ the category of multi\nobreakdash-sorted $\Mcc$-algebraic theories
(cf. Section~\ref{sect:ch_colour}).

Let $S$ be a fixed set of sorts. The category $\Mcc$-$\At_S$ has as objects the $S$-sorted $\Mcc$-algebraic theories and as morphisms the morphisms of algebraic theories that are the identity on $S$.

Let $\ECat{\Mcc}$ be the category of small $\Mcc$-enriched categories. There is a forgetful functor
\[
\lmorp{u}{\EAt{\Mcc}}{\ECat{\Mcc}}
\]
that sends each algebraic theory to its underlying category. By abuse of notation, we denote $u(\at{T})$ simply by $\at{T}$ when no confusion can arise.

Given a $\Mcc$-algebraic theory $\at{T}$, the category of $\at{T}$-algebras in $\Mcc$, denoted by $\EAlg{\Mcc}(\at{T})$ is the full subcategory of $\Mcc^{u(\at{T})}$ spanned by the product preserving $\Mcc$-functors.

\subsection{Morita equivalences of multi-sorted algebraic theories}

We now focus on the multi-sorted algebraic theories in $\sets$. The category $\EAt{\sets}$ will be simply denoted by $\At$, and for every $\at{T}$ in $\At$, its category of $\at{T}$-algebras in sets will be denoted by $\Alg(\at{T})$. The category $\Alg(\at{T})$ is a reflective subcategory of $\sets^{\at{T}}$. In other words, the inclusion functor $\Na\colon\Alg(\at{T}) \to \sets^{\at{T}}$ has a left adjoint
\[
\Sa: \sets^{\at{T}} \myrightleftarrows{\rule{0.4cm}{0cm}}\Alg(\at{T}): \Na.
\]

Given a morphism of algebraic theories $f\colon{\at{S}}\to{\at{T}}$, the extension-restriction adjunction associated to the underlying functor $u(f)$, restricts to an adjunction between the respective categories of algebras. The fact that the right adjoint $f^*$ sends product preserving functors in $\sets^{\at{T}}$ to product preserving functors in $\sets^{\at{S}}$ follows since $f$ is product preserving. The left adjoint $f_!$ is defined as the left Kan extension along $f$ as represented in the following diagram
$$
\xymatrix{
\at{S} \ar[r]^f \ar[d]_F & \at{T}\ar[dl]^{j_!F} \\
\sets. &
}
$$
Since $\sets$ is cartesian closed and both $\at{S}$ and $\at{T}$ have finite products, if a functor $F$ preserves finite products, so does the left Kan extension $j_!F$. This shows that $f_!$ also restricts to the categoris of algebras. Thus, we have the following commutative diagram of adjunctions:
\begin{equation}\label{eq:ext-rest}
\xymatrix{
    \sets^{\at{S}}
    \ar@<3pt>[r]^-{f_!}
    \ar@<-3pt>[d]_-{\Sa} &
    \sets^{\at{T}}
    \ar@<3pt>[l]^-{f^*}
    \ar@<-3pt>[d]_-{\Sa} \\
    \Alg(\at{S})
    \ar@<3pt>[r]^-{f_!}
    \ar@<-3pt>[u]_-{\Na}&
    \Alg(\at{T}).
    \ar@<3pt>[l]^-{f^*}
    \ar@<-3pt>[u]_-{\Na}
}
\end{equation}

We now define Morita equivalences in $\At$ and characterize them as the morphisms inducing an equivalence at the level of algebras.
Recall that a functor $f\colon \mathcal{C} \to \mathcal{D}$ between small categories is called a Morita equivalence if and only if it is fully faithful and essentially surjective up to retracts (that is, for every $y$ in $\mathcal{D}$ there exist $x$ in $\mathcal{C}$ such that $y$ is isomorphic to a retract of $f(x)$); see~\cite[Section 1.2]{CG19}. 

\begin{defi}
A map $f$ of algebraic theories is called a \emph{Morita equivalence} if the underlying functor $u(f)$ is a Morita equivalence of categories.
\end{defi}

Recall that given an algebraic theory $\at{T}$, the \emph{category of $\at{T}$-algebras} $\Alg({\at{T}})$ is the full subcategory of $\sets^{\at{T}}$ spanned by the product preserving functors. As in the case of Morita equivalences of categories, the Morita equivalences of algebraic theories can be characterized in terms of their categories of algebras.

\begin{theo}\label{theo:moritaAtSet}
Let $f\colon\at{S}\to \at{T}$ be a morphism between algebraic theories. The following are equivalent:
\begin{itemize}
    \item[{\rm (i)}] The morphism $f$ is a Morita equivalence.
    \item[{\rm (ii)}] The extension-restriction adjunction $f_!\colon \Alg(\mathcal{S})\rightleftarrows \Alg(\mathcal{T}):f^*$ is an equivalence of categories.
    \end{itemize}
\end{theo}
\begin{proof}
The fact that (i) implies (ii) follows from Theorem \ref{thm:moritaCat} applied to $u(f)$ and from diagram (\ref{eq:ext-rest}). If the top horizontal extension-restriction adjunction in (\ref{eq:ext-rest}) is an equivalence then the adjunction between algebras is also an equivalence. 

For the converse, consider the following commutative diagram
$$
\xymatrix{
  \at{S}^{\op}\ar[r]^{f^{\op}} \ar[d] & \at{T}^{\op} \ar[d] \\
\Alg(\at{S}) \ar[r]_{f_!} & \Alg(\at{T}),
}
$$ 
where the vertical arrows denote the Yoneda embeddings. Since $f_!$ and the Yoneda embeddings are fully faithful, so is $f^{\op}$. To prove that $f^{\op}$ is essentially surjective up to retracts consider an object $\str{a}$ in $\at{T}$. The functor $f_!$ is essentially surjective by assumption, so there exists $X$ in $\Alg(\at{S})$ such that $f_!X\cong\at{T}(\str{a},-)$. Now, we can check using the fully faithfulnes of $f_!$ that the functor $\Alg(\at{S})(X, -)$ preserves all colimits. Since every $\at{S}$-algebra is a colimit of corepresentables, if $X\cong\colim_i \at{S}(\str{a}_i,-)$, then  
$$
\Alg(\at{S})(X,\colim_i \at{S}(\str{a}_i,-))\cong \colim_i\Alg(\at{S})(X, \at{S}(\str{a}_i,-)),
$$
and taking the identity in $\Alg(\at{S})(X,X)$ we can show that $X$ is a retract of a corepresentable $\at{S}(\str{a}_j, -)$ for some $j$. Applying $f_!$ to this retraction and using Yoneda lemma we obtain that $\str{a}$ is the retract of $f(\str{a}_j)$, which is in the image of $f^{\op}$. We finish the proof by observing that $f$ is a Morita equivalence if and only if $f^{\op}$ is a Morita equivalence.
\end{proof}
 
\subsection{Simplicial algebraic theories and their algebras}

We now pass to the homotopical case and focus therefore on simplicial algebraic theories.
We will denote the category of algebraic theories enriched in $\ssets$ by $\sAt$.
Recall that for every simplicial algebraic theory $\at{T}$, its category of algebras in $\ssets$, denoted by $\sAlg{\at{T}}$, is the full subcategory of $\ssets^{\at{T}}$ spanned by the product preserving \emph{simplicial} functors. As in the non-enriched case, the inclusion $\morp{\Na}{\sAlg{\at{T}}}{\ssets^{\at{T}}}$ that sends a simplicial $\at{T}$-algebra $X$ to
$\sAlg{h_{(-)}, X}\cong X$, where $h$ denotes the corepresentable functor, has a left adjoint that we will denote by $\Sa$.

Let $\at{T}$ be an $S$-sorted simplicial algebraic theory. For every object $\str{a}$ in $\at{T}$ of the form $\langle a_1,\ldots, a_n\rangle$, let $\crep{\str{a}}$ be the corepresentable functor $\at{T}(\str{a},-)$ in $\ssets^{\at{T}}$, and let $p_{\str{a}}$ be the canonical map
\[
\lmorp{p_{\str{a}}}{\crep{a_1} \sqcup \dots \sqcup \crep{a_n}}{\crep{\str{a}}.}
\]
Observe that $\crep{\str{a}}$ is product preserving, that is, it is in the essential image of $\Na$.

We denote by $L_{\at{T}} =\{p_{\str{a}} \mid \str{a}\in \Str{S}\}$ the set of all such maps.
A simplicial functor $F$ in $\ssets^{\at{T}}$ is in the (essential) image of $\Na$ if and only if it is orthogonal to $L_{\at{T}}$, that is, if and only if the induced map
\[
\lmorp{(p_{\str{a}})_*}{\ssets^{\at{T}}(\crep{\str{a}}, F)\cong F(\str{a})}{ F(a_1)\times \dots \times F(a_n)\cong \ssets^{\at{T}}(\crep{a_1} \sqcup \dots \sqcup \crep{a_n}, F)}
\]
is an isomorphism for every $\str{a}$ in $\Str{S}$.

Given a morphism of simplicial algebraic theories $f\colon{\at{S}}\to{\at{T}}$, the extension-restriction adjunction associated to the simplicial functor $u(f)$, restricts to an adjunction between the corresponding categories of simplicial algebras, as represented by the following commutative diagram of adjunctions:

\[
\xymatrix{
    \ssets^{\at{S}}
    \ar@<3pt>[r]^-{f_!}
    \ar@<-3pt>[d]_-{\Sa} &
    \ssets^{\at{T}}
    \ar@<3pt>[l]^-{f^*}
    \ar@<-3pt>[d]_-{\Sa} \\
    \sAlg{\at{S}}
    \ar@<3pt>[r]^-{f_!}
    \ar@<-3pt>[u]_-{\Na}&
    \sAlg{\at{T}}.
    \ar@<3pt>[l]^-{f^*}
    \ar@<-3pt>[u]_-{\Na}
}
\]

\subsection{The homotopy theory of algebras over an algebraic theory}

We now recall some facts about the homotopy theory of algebras over a simplicial algebraic theory. From now on the category $\ssets$ will be always considered equipped with the Kan--Quillen model structure, that models the homotopy theory of spaces.

The goal for the upcoming sections is to prove Corollary \ref{cor:char_Morita_at}, providing a characterization of the morphisms of simplicial algebraic theories inducing a homotopical equivalence between the corresponding categories of algebras.

In \cite{BA02} (for one-sorted algebraic theories) and \cite{Ber06} (for multi-sorted algebraic theories) two models are presented for the homotopy theory of algebras over a (simplicial) algebraic theory $\mathcal{T}$: a model structure on the category of algebras (\emph{strict models}), and a model structure obtained as localization of the projective model structure on $\ssets^{\mathcal{T}}$, in which the fibrant objects are the \emph{weak models} for~$\mathcal{T}$. In these papers, Badzioch and Bergner proved the equivalence of these two approaches.

In order to prove Corollary \ref{cor:char_Morita_at}, we roughly proceed as follows. First we show that if $f\colon \mathcal{S} \to \mathcal{T}$ in Theorem~\ref{thm:DK} is a map of algebraic theories (not just of categories), then the equivalence restricts to an equivalence between the weak models of the two theories (Theorem~\ref{thm:char_Morita}). We then exploit the above mentioned equivalence between the weak models and the strict models to conclude that $f$ induces a Quillen equivalence between the corresponding categories of algebras.

Let $\at{T}$ be an $S$\nobreakdash-sorted simplicial algebraic theory. The projective model structure on $\ssets^{\at{T}}$ can be transferred to a model structure on $\sAlg{\at{T}}$ along $\Na$ to model the homotopy theory of simplicial algebras over $\at{T}$; see~\cite[Theorem~7.1]{Re02}. We will always consider $\sAlg{\at{T}}$ with this model structure.

\begin{theo}[Rezk]
Let $\at{T}$ be an $S$-sorted simplicial algebraic theory. Then the category of algebras $\sAlg{\at{T}}$ admits a right proper simplicial model structure in which a map $X\to Y$ is a weak equivalence (respectively a fibration) if and only if for every $\str{a}$ in $\Str{S}$, the map $X(\str{a})\to Y(\str{a})$ is a weak equivalence (respectively a fibration) of simplicial sets.
\end{theo}

Given a map of simplicial algebraic theories $f\colon \at{S}\to \at{T}$, the extension-restriction adjunction
        \[
        f_!:\sAlg{\at{S}}\myrightleftarrows{\rule{0.4cm}{0cm}} \sAlg{\at{T}}: f^*
        \]
is a Quillen pair.

A model for the theory of homotopy algebras over $\at{T}$ can be obtained by left Bousfield localizing the projective model structure on $\ssets^{\at{T}}$ with respect to the set of maps $L_{\at{T}} =\{p_{\str{a}} \mid \str{a}\in \Str{S}\}$; see~\cite[Proposition 4.9]{Ber06} and \cite[Section 5]{BA02}. We will denote this localization by $\ssets^{\at{T}}_{\wm}$ and we will call its weak equivalences \emph{$L_\at{T}$\nobreakdash-local equivalences}.
\begin{defi}
    Let $\at{T}$ be an $S$\nobreakdash-sorted algebraic theory. A functor $X$ in $\ssets^{\at{T}}$ is
    called a \emph{weak model} or a \emph{homotopy $\at{T}$-algebra} if for every $\str{a}=\langle a_1,\dots,a_n\rangle$ in $\Str{S}$ the induced map of simplicial sets
    \[
    (p_{\str{a}})_*\colon X(\str{a})\longrightarrow X(a_1)\times \dots \times X(a_n)
    \]
    is a weak equivalence.
\end{defi}
The fibrant objects in $\ssets^{\at{T}}_{\wm}$ are precisely the functors that are weak models and
levelwise fibrant~\cite[Proposition 4.10]{Ber06}. Badzioch (for ordinary algebraic theories) and Bergner (in the multi-sorted case) proved that the model structures  $\sAlg{\at{T}}$ and $\ssets^{\at{T}}_{\wm}$ are Quillen equivalent; see~\cite[Theorem~6.4]{BA02} and \cite[Theorem~5.1]{Ber06}.

\begin{theo}
[Badzioch, Bergner]\label{thm:ber_bad}
The adjunction $\Sa:\ssets^{\at{T}}_{\wm}\rightleftarrows \sAlg{\at{T}}: \Na$ is a Quillen equivalence.\end{theo}

\begin{rem}
The theorem above is, in fact, proved by Badzioch and Bergner for algebraic theories enriched in sets and not in simplicial sets. However, as mentioned in~\cite[Note 1.5]{BA02}, the same proof is valid for simplicial algebraic theories.
\end{rem}

For every map $f\colon \at{S}\to\at{T}$ of simplicial algebraic theories we have a commutative diagram of Quillen pairs
\begin{equation}\label{diag:adjunctions}
\xymatrix{
    \ssets^{\at{S}}
    \ar@<3pt>[r]^-{\id}
    \ar@<-3pt>[d]_-{f_!} &
    \ssets^{\at{S}}_{\wm}
    \ar@<3pt>[r]^-{\Sa}
    \ar@<-3pt>[d]_-{f_!}
    \ar@<3pt>[l]^-{\id} &
    \sAlg{\at{S}}
    \ar@<-3pt>[d]_-{f_!}
    \ar@<3pt>[l]^-{\Na}  \\
    \ssets^{\at{T}}
    \ar@<3pt>[r]^-{\id}
    \ar@<-3pt>[u]_-{f^*} &
    \ssets^{\at{T}}_{\wm}
    \ar@<3pt>[r]^-{\Sa}
    \ar@<3pt>[l]^-{\id}
     \ar@<-3pt>[u]_-{f^*} &
    \sAlg{\at{T}}
    \ar@<-3pt>[u]_-{f^*}
    \ar@<3pt>[l]^-{\Na}
}
\end{equation}

Recall that the path component functor $\pi_0\colon \ssets \to \sets$ extends to a functor $\pi_0\colon \sCat \to \Cat$ such that $\pi_0(\mathcal{C})(x, y) = \pi_0(\mathcal{C}(x,y))$ for every simplicial category $\mathcal{C}$ and every $x,y\in \mathcal{C}$.
   
\begin{defi}
Let $f\colon \at{S} \to \at{T}$ be a morphism of simplicial algebraic theories. We say that
\begin{itemize}
    \item[{\rm (i)}] $f$ is \emph{homotopically fully faithful} if $u(f)$ is homotopically fully faithful, that is, if $f\colon \at{S}(\mathbf{c}, \mathbf{c'}) \to \at{T}(f(\mathbf{c}), f(\mathbf{c'}))$ is a weak equivalence for every $\mathbf{c}, \mathbf{c'} \in \at{S}$;
    \item[{\rm (ii)}] $f$ is \emph{homotopically essentially surjective up to retracts} if the functor $\pi_0(u(f))$ is essentially surjective up to retracts;
    \item[{\rm (iii)}] $f$ is a \emph{Morita weak equivalence} if it is homotopically fully faithful and homotopically essentially surjective up to retracts.
   \end{itemize} 
\end{defi}

Our goal is to prove that the adjunction $(f_!,f^*)$ on the right in~(\ref{diag:adjunctions}) is a Quillen equivalence if and only if $f$ is a Morita weak equivalence of algebraic theories. Thanks to Theorem~\ref{thm:ber_bad}, we can equivalently prove that $f$ is a Morita weak equivalence if and only if the vertical adjunction in the middle of~(\ref{diag:adjunctions}) is a Quillen equivalence. We will show this in Corollary~\ref{cor:char_Morita_at}.

\subsection{Homotopy left Kan extensions and weak models}
Let $\cat{C}$ be a small simplicial category. Given a finite sequence of objects $\str{c}=\{c_i\}_{i=0}^n$ of $\cat{C}$, we will denote by $\nervpr{\cat{C}}{\str{c}}$ the simplicial set $\prod_{i=0}^{n-1} \cat{C}(c_i,c_{i+1})$.

Consider two simplicial functors $X$ and $Y$ in $\ssets^{\cat{C}}$ and $\ssets^{\cat{C}\op}$, respectively.
The \emph{(two sided) bar construction} of $X$ and $Y$ is the bisimplicial set $\bc{X}{\cat{C}}{Y}$ which is defined as
\[
\bcn{X}{\cat{C}}{Y}{n}_m= \underset{\str{c}=\{c_i\}_{i=0}^n\in \cat{C}^{n+1}}\coprod X(c_0)_m \times \nervpr{\cat{C}}{\str{c}}_m \times Y(c_n)_m
\]
for every $m,n\ge 0$. Thus, an element of $\bcn{X}{\cat{C}}{Y}{n}_m$ is represented by a tuple $(\str{c}\mid x; \{\alpha_i\}_{i=0}^{n-1}; y)$, where:
\begin{itemize}
    \item $\str{c}=(c_0,\dots,c_n)\in \cat{C}^{n+1}$;
    \item $x$ is an element of $X(c_0)_m$;
    \item $\{\alpha_i\}_{i=0}^{n-1}$ belongs to $\nervpr{\cat{C}}{\str{c}}_m$, that is, $\alpha_i\in \cat{C}(c_i, c_{i+1})_m$ for every $0\le i\le n-1$;
    \item $y$ is an element of $Y(c_n)_m$.
\end{itemize}
We will denote by the symbol $\dbc{X}{\cat{C}}{Y}$ the simplicial set obtained as the diagonal of the bisimplicial set $\bc{X}{\cat{C}}{Y}$.

Let $f\colon\mathcal{C}\to\mathcal{D}$ be a map of simplicial categories and consider the right $\mathcal{C}$\nobreakdash-module
\[
\gfun{\mathcal{D}(f(-),-)}{\mathcal{C}\op}{\ssets^{\mathcal{D}}}{c}{\mathcal{D}(f(c),-).}
\]
It is well known that for every left $\mathcal{C}$\nobreakdash-module $X$ in $\ssets^{\mathcal{C}}$, the left Kan extension of $X$ along $f$, denoted by $f_!(X)$, is isomorphic to the coend
\[
f_!(X)\cong\int^{c\in \mathcal{C}} X(c) \times \mathcal{D}(f(c),-).
\]

Let $\bcf{f}{X}$ in $(\ssets^{\Delta\op})^{\cat{D}}$ be the functor that assigns to each $d$ in $\mathcal{D}$ the bar construction $\bc{X}{\cat{C}}{\cat{D}(f(-),d)}$.
The homotopy left Kan extension of $X$ along $f$, denoted by $\tilde{f_!}(X)$, can be computed as the diagonal of $\bcf{f}{X}$. In other words,
\[
\begin{split}
&\tilde{f_!}(X)(d)_n\cong \bcf{f}{X}(d)_{n,n}\cong \dbcn{X}{\cat{C}}{\cat{D}(f(-),d)}{n}\cong \\
& \underset{c_0,\dots,c_n\in \cat{C}} \coprod X(c_0)_n\times \mathcal{C}(c_0,c_1)_n \times\dots \times \mathcal{C}(c_{n-1},c_n)_n \times \mathcal{D}(f(c_n),d)_n,
\end{split}
\]
for every $d$ in $\cat{D}$ and $n\ge 0$.
This assignment extends to a (simplicial) functor
$\morp{\tilde{f_!}}{\ssets^{\cat{C}}}{\ssets^{\cat{D}}}$.
The functor $\tilde{f_!}$ has a right adjoint $\tilde{f^*}$ defined as
\[
\tilde{f^*}(Y)(c) \cong \mathcal{\ssets^{\mathcal{D}}}(\tilde{f_!}(\mathcal{C}(c,-)), Y)
\]
for every $Y$ in $\ssets^{\mathcal{D}}$ and $c$ in $\mathcal{C}$; see \cite[\S 3]{DK87} for more details about the construction of the functors $\tilde{f_!}$ and $\tilde{f^*}$.
The following result can be found in~\cite[\S 3.5--\S 3.8]{DK87} (cf. \cite[Ch.IX.2]{GJ99})
\begin{propo}[Dwyer--Kan]\label{prop:deform}
Let $f\colon \mathcal{C}\to \mathcal{D}$ be a map of simplicial categories and let $\tilde{f_!}: \ssets^{\mathcal{C}}\rightleftarrows \ssets^{\mathcal{D}}:\tilde{f^*}$ be the adjunction given by homotopy
left Kan extension. The following holds:
\begin{itemize}
\item[{\rm (i)}] There exists a natural weak equivalence from $f^*$ to $\tilde{f^*}$.
\item[{\rm (ii)}] The adjunction $(\tilde{f_!},\tilde{f^*})$ is a Quillen pair between $\ssets^{\mathcal{C}}$ and $\ssets^{\mathcal{D}}$ with the projective model structures. Furthermore, both $\tilde{f_!}$ and $\tilde{f^*}$ preserve weak equivalences.
\item[{\rm (iii)}] The adjunction $(f_!, f^*)$ is a Quillen equivalence if and only if $(\tilde{f_!},\tilde{f^*})$ is a Quillen equivalence.
\end{itemize}
\end{propo}

\begin{rem}
We warn the reader that the notation we use for the homotopy left Kan extension and its right adjoint differs from the one used in \cite{DK87} and \cite{GJ99}. While Goerss--Jardine use $\tilde{f^*}$ and $\tilde{f_*}$ for what we denote by $\tilde{f_!}$ and $\tilde{f^*}$, Dwyer--Kan use $\underrightarrow{f_*}$ for the homotopy left Kan extension and $\underrightarrow{f^*}$ for its right adjoint.
\end{rem}

Let $\at{T}$ be a simplicial algebraic theory and suppose that $g\colon X\to Y$ is a weak equivalence in $\ssets^{\at{T}}$ with the projective model structure. Then $X$ is a weak model if and only if $Y$ is a weak model, since weak equivalences are closed under taking products.

\begin{propo}\label{prop:webetweenwm}
Let $\at{T}$ be a simplicial algebraic theory. A map $g\colon X\to Y$ between weak models in $\ssets^{\at{T}}$ is an $L_{\at{T}}$\nobreakdash-local equivalence if and only if it is a projective weak equivalence.
\end{propo}
\begin{proof}
 Using the factorization axioms of model categories in $\ssets^{\at{T}}$ with the projective model structure we can find a commutative diagram
    \[
    \xymatrix{X \ar[r]^-{g}
        \ar[d]_-{\sim} &
        Y
        \ar[d]^-{\sim} \\
        \widehat{X} \ar[r]_-{\widehat{g}} &
        \widehat{Y}
    }
    \]
such that $\widehat{X}$ and $\widehat{Y}$ are fibrant in $\ssets^{\at{T}}$. Since $X$ and $Y$ are weak models, so are $\widehat{X}$ and $\widehat{Y}$. Therefore $\widehat{X}$ and $\widehat{Y}$ are fibrant in $\ssets^{\at{T}}_{\wm}$. Since $\ssets^{\at{T}}_{\wm}$ is a left Bousfield localization of the projective model structure on $\ssets^{\at{T}}$, the map $\widehat{g}$ is an $L_{\at{T}}$\nobreakdash-local equivalence if and only if it is a projective weak equivalence. The vertical maps are projective weak equivalences, so $g$ is an $L_{\at{T}}$\nobreakdash-local equivalence if and only if it is a projective weak equivalence.
\end{proof}

\begin{propo}\label{prop:weaktoweak}
For every morphism $f$ of simplicial algebraic theories the functor $\tilde{f^*}$ sends weak models to weak models.
\end{propo}
\begin{proof}
    Since $f$ is product preserving, the functor $f^*$ preserves weak models. The statement then follows from Proposition \ref{prop:deform}(i).
\end{proof}

\begin{coro}
For every  morphism $f\colon \at{S}\to \at{T}$ of simplicial algebraic theories the adjunction
$\tilde{f_!}: \ssets^{\mathcal{S}}_{\wm}\rightleftarrows \ssets^{\mathcal{T}}_{\wm}: \tilde{f^*}$ is a Quillen pair.
\end{coro}
\begin{proof}
It is enough to prove that $\tilde{f_!}$ preserves cofibrations and $\tilde{f^*}$ sends weak equivalences between fibrant objects to weak equivalences. Since $\ssets^{\mathcal{S}}_{\wm}$ is a left Bousfield localization it has the same cofibration as $\ssets^{\mathcal{S}}$, so $\tilde{f_!}$ clearly preserves cofibrations.

Let $g$ be an $L_{\at{T}}$-local equivalence between fibrant objects. Since fibrant objects in
$\ssets^{\mathcal{T}}_{\wm}$ are in particular weak models, $g$ is a projective weak equivalence, by Proposition~\ref{prop:webetweenwm}. Hence $\tilde{f^*}(g)$ is a projective weak equivalence in $\ssets^{\mathcal{S}}$ by Proposition~\ref{prop:deform}(ii). But since $\tilde{f^*}$ preserves weak models, by Proposition~\ref{prop:weaktoweak}, $\tilde{f^*}(g)$ is an $L_{\at{S}}$-local equivalence, again by Proposition~\ref{prop:webetweenwm}.
\end{proof}

\begin{rem}\label{rem:adj_tilde_f}
Since by Proposition~\ref{prop:deform} the functor $f^*$ is weakly equivalent to $\tilde{f}^*$, it follows that for every morphism $f\colon\mathcal{S}\to \mathcal{T}$ of simplicial algebraic theories the adjunction $\tilde{f_!}: \ssets^{\mathcal{S}}_{\wm}\rightleftarrows \ssets^{\mathcal{T}}_{\wm}: \tilde{f^*}$ is a Quillen equivalence if and only if ${f_!}: \ssets^{\mathcal{S}}_{\wm}\rightleftarrows \ssets^{\mathcal{T}}_{\wm}:{f^*}$ is a Quillen equivalence.
\end{rem}

\begin{propo}\label{prop:weakmod}
The functor $\tilde{f_!}$ sends weak models to weak models for every map $f$ of simplicial algebraic theories.

\end{propo}
\begin{proof}
Let $f\colon \at{S}\to \at{T}$ be a map of simplicial algebraic theories. It suffices to prove that if $X$ is a weak model for $\at{S}$, then the canonical map
$$
p\colon {\tilde{f_!}(X)(\str{a}\str{b})}\longrightarrow{\tilde{f_!}X(\str{a})\times \tilde{f_!}X(\str{b})}
$$
is a weak equivalence for every $\str{a}, \str{b}$ in $\at{T}$.

We can consider two bisimplicial objects associated to $X$, namely:
\[
\bcf{f}{X}(\str{a}\str{b}) =\bc{X}{\at{S}}{\at{T}(f(-),\str{a}\str{b})}
\]
and
\[
\bcf{f}{X}(\str{a})\times \bcf{f}{X}(\str{b})= \bc{X}{\at{S}}{\at{T}(f(-),\str{a})}\times \bc{X}{\at{S}}{\at{T}(f(-),\str{b})}.
\]
We are also going to consider the following auxiliary bisimplicial object:
\[
\bcf{(f\times f)}{X}(\str{a},\str{b}) = \bc{X(-\times -)}{\at{S}\times \at{S}}{\at{T}(f(-),\str{a}) \times \at{T}(f(-),\str{b})}.
\]
Given two finite sequences of the same length $\str{u}=\{\str{u}_i\}_{i=0}^n$ and $\str{v}=\{\str{v}_i\}_{i=0}^n$ of objects in $\at{S}$ we will denote by $(\str{u},\str{v})$ the sequence $\{(\str{u}_i,\str{v}_i)\}_{i=0}^n$ of objects of $\at{S}\times \at{S}$ and by $\str{u}\str{v}$ the sequence $\{\str{u}_i\str{v}_i\}_{i=0}^n$ in $\at{S}$.
There are natural morphisms
\[
\xymatrix{
    & \ar[ld]_-{\psi} \bcf{(f\times f)}{X
    }(\str{a},\str{b}) \ar[rd]^-{\phi} \\
    \bcf{f}{X
    }(\str{a}\str{b}) \ar[rr]_-{\delta} & & \bcf{f}{X
}(\str{a}) \times \bcf{f}{X
}(\str{b}).}
\]
Since the diagonal of $\delta$ is isomorphic to $p$, in order to prove our statement it is sufficient to prove that $\delta$ is a levelwise weak equivalence, by the realization lemma for bisimplicial sets.

The morphism $\phi$ is a levelwise weak equivalence, since $X$ is a weak model. By the two out of three property it is then enough to show that $\psi$ is a levelwise weak equivalence.

For every $n,m \in \mathbb{N}$ the set $((\overline{{f\times f})(X)}(\str{a},\str{b}))_{m,n}$ is isomorphic to
\[
\underset{\str{v}_0,\dots,\str{v}_n\in \at{S}} {\underset{\str{u}_0,\dots,\str{u}_n\in \at{S}} \coprod}\hspace{-0.3cm}
X(\str{u}_0 \str{v}_0)_m \times \nervpr{(\at{S}\times \at{S})}{\str{u},\str{v}}_m
\times \at{T}(f(\str{u}_n),\str{a})_m\times \at{T}(f(\str{v}_n),\str{b})_m,
\]
and the set $(\bcf{f}{X}(\str{a}\str{b}))_{m,n}$ is isomorphic to
\[
{\underset{\str{u}_0,\dots,\str{u}_n\in \at{S}} \coprod}
X(\str{u}_0)_m \times \nervpr{\at{S}}{\str{u}}_m
\times \at{T}(f(\str{u}_n),\str{a})_m\times \at{T}(f(\str{u}_n),\str{b})_m.
\]
The morphism $\psi_{m,n}$ sends $(\str{u},\str{v}\mid x; \{\alpha_i,\beta_i\}_{i=0}^{n-1} ; g,h)$ in $((\overline{{f\times f})(X)}(\str{a},\str{b}))_{m,n}$ to
$(\str{u}\str{v}\mid x; \{\alpha_i\times\beta_i\}_{i=0}^{n-1}  ; g\pi_1, h\pi_2)$.

There is also a morphism $\morp{\sigma}{\bcf{f}{X
    }(\str{a}\str{b})}{\bcf{(f\times f)}{X
}(\str{a},\str{b})}$ in the opposite directions, sending $(\str{u}\mid x; \{\gamma_i\}_{i=0}^{n-1} ; g,h)$ to $(\str{u},\str{u}\mid \Delta^*(x); \{\gamma_i, \gamma_i\}_{i=0}^{n-1}, ; g,h)$.

We are going to prove that $\sigma_{m\bullet}$ is a homotopy inverse for $\psi_{m\bullet}$ for every $m\in \mathbb{N}$.
We begin by exhibiting a simplicial homotopy $J$ from $\id$ to $(\sigma\psi)_{m\bullet}$.
Recall that a simplicial homotopy between two maps of simplicial sets $\morp{f,g}{X}{Y}$ can be defined by giving for every $n\in \N$ a set of functions $\{\morp{H^n_i}{X_n}{Y_{n+1}}\}_{i=0}^{n}$  satisfying certain relations (see, for example, \cite[\S 5]{May67} for details).
In particular, $\delta_0 H^n_0 = f_n$ and $\delta_{n+1} H^n_{n} = g_n$.

The evaluation of $(\sigma\psi)_{m\bullet}$ on an $n$\nobreakdash-simplex $(\str{u},\str{v}\mid x; \{\alpha_i,\beta_i\}_{i=0}^{n-1} ; g,h)$ is
\[
\begin{split}
(\sigma \psi)_{m,n}(\str{u},\str{v}\mid x; \{\alpha_i,\beta_i\}_{i=0}^{n-1} ; g,h) = (\str{u}\str{v},\str{u}\str{v}\mid x; \{\alpha_i\times \beta_i, \alpha_i\times \beta_i\}_{i=0}^{n-1}; g\pi_1, h\pi_2).
\end{split}
\]
For every sequence $\str{u}$ in $\at{S}^{n+1}$ and $0 \leq j\leq n$ we define three objects $\str{u}^{\star j}$, $\str{u}^{|j}$ and $\str{u}^{j|}$ in $\at{S}^{n+2}$ as follows:
\[
\str{u}^{\star j}_i=
\begin{cases}
\str{u}_i & i\leq j,\\
\str{u}_{i-1} & i> j,
\end{cases}
\qquad
\str{u}^{|j}_i=
\begin{cases}
\str{u}_i & i\leq j,\\
\ast & i> j,
\end{cases}
\qquad
\str{u}^{j|}_i=
\begin{cases}
\ast  & i\leq j,\\
\str{u}_{i-1} & i> j.
\end{cases}
\]
The homotopy $J$ is defined as follows: for every $n\in \N$, every $0\leq j \leq n$, and every $n$\nobreakdash-simplex $(\str{u},\str{v}\mid x; \{\alpha_i,\beta_i\}_{i=0}^{n-1} ; g,h)$ the $n+1$ simplex $J^n_j(\str{u},\str{v}\mid x; \{\alpha_i,\beta_i\}_{i=0}^{n-1} ; g,h)$ is equal to
\newcommand{\nartimes}{\hspace{-0.05cm}\times \hspace{-0.07cm}}
\[
\begin{split}
(\str{u}^{\star j}\str{v}^{|j},\str{u}^{|j}\str{v}^{\star j}\mid \Delta^*(x); (\alpha_0\nartimes \beta_0, \alpha_0\nartimes \beta_0),\dots, (\pi_1,\pi_2), \alpha_j\nartimes \beta_j, \dots,   \alpha_{n-1}\nartimes \beta_{n-1}; g, h).
\end{split}
\]
The collection $\{J^n_j\}_{n\in \N, 0\leq j \leq n}$ defines a simplicial homotopy from $(\sigma\psi)_{m\bullet}$ to the identity.

To conclude we have to exhibit a homotopy $K$ from $(\psi \sigma)_{m\bullet}$ to the identity.
For every $n$\nobreakdash-simplex $(\str{u}\mid x; \{\gamma_i\}_{i=0}^{n-1} ; g,h)$ we have that
\[
\begin{split}
(\psi\sigma)_{m,n}(\str{u}\mid x; \{\gamma_i\}_{i=0}^{n-1} ; g,h) = (\str{u}\str{u}\mid \Delta_*(x) ; \{\gamma_i\times \gamma_i\}_{i=0}^{n-1}; g\pi_1, h\pi_2).
\end{split}
\]
The homotopy $K$ is defined by the requirement that for every $n\in\N$, every $0\leq j\leq n$ and every $n$\nobreakdash-simplex $(\str{u}\mid x; \{\gamma_i\}_{i=0}^{n-1} ; g,h)$, the $n+1$\nobreakdash-simplex  $K^n_j(\str{u}\mid x; \{\gamma_i\}_{i=0}^{n-1} ; g,h)$ is equal to
\[
(\str{u}^{j|}\str{u}\mid x ; \gamma_0, \gamma_1, \dots, \Delta, \gamma_{j+1} \times \gamma_{j+1}, \dots,\gamma_{n-1} \times \gamma_{n-1}; g\pi_1, h\pi_2).
\]
It is easy to check that $K$ is indeed a well-defined homotopy.
\end{proof}
\begin{coro}
For every morphism $f\colon \at{S}\to \at{T}$ of simplicial algebraic theories the functor $\tilde{f_!}$ sends $L_\at{S}$\nobreakdash-local equivalences between weak models to $L_\at{T}$\nobreakdash-local equivalences.
\end{coro}
\begin{proof}
 Since $\tilde{f_!}$ preserves projective weak equivalences by Proposition~\ref{prop:deform}(ii), the result follows from Proposition \ref{prop:webetweenwm} and Proposition \ref{prop:weakmod}.
\end{proof}

\subsection{Characterization of Morita weak equivalences of algebraic theories}

    We are finally ready to prove that a morphism of simplicial algebraic theories induces a Quillen equivalence between the respective homotopy theories of simplicial algebras if and only if it is a Morita weak equivalence.

\begin{lemma}\label{lemma:quequnit}
Let $\morp{f}{\at{S}}{\at{T}}$ be a morphism of simplicial algebraic theories. If the Quillen pair
$\tilde{f_!}: \ssets^{\at{S}}_{\wm}\rightleftarrows \ssets^{\at{T}}_{\wm}:\tilde{f^*}$ is a Quillen equivalence, then:
\begin{itemize}
\item[{\rm (i)}] For every weak model $X$ in $\ssets^{\at{S}}$ the unit $\morp{\eta_X}{X} {\tilde{f^*}\tilde{f_!}X}$ is an $L_\at{S}$\nobreakdash-local equivalence.
\item[{\rm (ii)}] For every weak model $Y$ in $\ssets^{\at{T}}$ the counit $\morp{\varepsilon_Y}{\tilde{f_!}\tilde{f^*}Y}{Y}$ is an $L_\at{T}$\nobreakdash-local equivalence.
\end{itemize}
\end{lemma}
\begin{proof}
 We only prove part (i); the proof of part (ii) is analogous. Let $X$ be a weak model in $\ssets^{\at{S}}$and let $\widetilde{X}$ be a projective cofibrant replacement of it. Let $\widehat{\tilde{f_!}\widetilde{X}}$ and $\widehat{\tilde{f_!}X}$  be projective fibrant replacements of $\tilde{f_!}\widetilde{X}$ and $\tilde{f_!}X$, respectively. Consider the commutative diagram in $\ssets^{\at{S}}$:
    \[
    \xymatrix{
        X \ar[r]^-{\eta_X}  &
        \tilde{f^*}\tilde{f_!} X \ar[r]^-{\sim} &
        \tilde{f^*}\widehat{\tilde{f_!} X} \\
        \widetilde{X} \ar[r]_-{\eta_{\widetilde{X}}} \ar[u]^-{\sim}&
        \tilde{f^*}\tilde{f_!} \widetilde{X} \ar[r]_-{\sim} \ar[u]_-{\sim}&
        \tilde{f^*}\widehat{\tilde{f_!} \tilde{X}, \ar[u]_-{\sim}}
    }
    \]
where all the arrows except $\eta_X$ and $\eta_{\widetilde{X}}$ are projective weak equivalences, by Proposition~\ref{prop:deform}(ii).
    If $(\tilde{f_!}, \tilde{f^*})$ is a Quillen equivalence, then the bottom horizontal composition is an $L_\at{S}$\nobreakdash-local equivalence, hence $\eta_{\widetilde{X}}$ is an $L_{\at{S}}$-local equivalence.  Note that $\widetilde{X}$ is a weak model since $X$ is so. By Proposition~\ref{prop:webetweenwm} it follows that $\eta_X$ and $\eta_{\widetilde{X}}$ are $L_\at{S}$\nobreakdash-local equivalences.
\end{proof}

Given a simplicial category $\mathcal{C}$ and an object $c$ of $\mathcal{C}$, we denote by $h_{c}$ the corepresentable functor $\mathcal{C}(c,-)\colon \mathcal{C}\to \ssets$. Recall from Section~\ref{sect:intro} that a morphism of simplicial categories is called a Morita weak equivalence if it is homotopically fully faithful and homotopically essentially surjective up to retracts.
\begin{propo}\label{prop:reqalg}
A morphism of simplicial categories $\morp{f}{\mathcal{C}}{\mathcal{D}}$ is a Morita weak equivalence if and only if the following two conditions hold:
\begin{itemize}
\item[{\rm(i)}]  For every $c$ in $\mathcal{C}$ the unit $\morp{\eta_c}{h_c}{\tilde{f^*}\tilde{f_!} h_c}$ is a projective weak equivalence in $\ssets^{\mathcal{C}}$.
\item[{\rm (ii)}] For every $d$ in $\mathcal{D}$ the counit $\morp{\varepsilon_d}{\tilde{f_!}\tilde{f^*}h_d}{h_d}$ is a projective weak equivalence in $\ssets^{\mathcal{D}}$.
\end{itemize}
\end{propo}
\begin{proof}
If $f\colon \mathcal{C}\to \mathcal{D}$ is a Morita weak equivalence, then $(\tilde{f_!},\tilde{f^*})$ is a Quillen equivalence and both $\tilde{f_!}$ and $\tilde{f^*}$ preserve weak equivalences, by Proposition~\ref{prop:deform}. It then follows that (i) and (ii) holds, not only for corepresentable functors but for all functors.

Conversely, suppose that (i) and (ii) hold. By~\cite[Proposition 3.3]{DK87}, for every $c$ in $\mathcal{C}$ the functor $\mathcal{D}(f(c),-)$ is a strong deformation retract of $\tilde{f}_!\mathcal{C}(c,-)$. This means that there exists maps $r\colon \tilde{f}_!\mathcal{C}(c,-)\to \mathcal{D}(f(c),-)$ and $s\colon \mathcal{D}(f(c),-)\to \tilde{f}_!\mathcal{C}(c,-)$ such that $r\circ s={\rm id}$ and $s\circ r\sim {\rm id}$. Then, the map
$$
i=s^*\circ\eta_c\colon \mathcal{C}(c, c')\longrightarrow \tilde{f}_!\mathcal{C}(c,-)(f(c'))
$$
is a weak equivalence for every $c'$ in $\mathcal{C}$. We have a commutative diagram
$$
\xymatrix{
\mathcal{C}(c, c') \ar[r]^-i\ar[dr]  & \tilde{f}_!\mathcal{C}(c,-)(f(c')) \ar[d]^r\\
 & \mathcal{D}(f(c), f(c')).
}
$$
Since $i$ and $r$ are weak equivalences, so is the map on the left. Hence, $f$ is homotopically fully faithful.

To show that $f$ is homotopically essentially surjective we use the same argument as in the proof of~\cite[Theorem 2.1]{DK87}.  Let $d$ be any object of $\mathcal{D}$ and consider the weak equivalence
$$
\tilde{f}_!f^*\mathcal{D}(d,-)(d)\longrightarrow \tilde{f}_!\tilde{f}^*\mathcal{D}(d,-)(d)\longrightarrow \mathcal{D}(d,d).
$$
Let $j\colon d\to f(c_0)$, $q\colon f(c_0)\to d$ be a $0$-simplex of $\tilde{f}_!f^*\mathcal{D}(d,-)(d)$ that is sent to the component of the identity by the previous map. Then $d$ is a retract of $f(c_0)$ in~$\pi_0(\mathcal{D})$.
\end{proof}
\begin{theo}\label{thm:char_Morita}
    Let $\morp{f}{\at{S}}{\at{T}}$ be a morphism of simplicial algebraic theories. Then the induced Quillen pair
    \[
    \xymatrix{
        {\tilde{f_!}\colon\ssets^{\at{S}}_{\wm}  \ar@<3pt>[r]} &
        \ssets^{\at{T}}_{\wm}\colon{\tilde{f^*}}	\ar@<3pt>[l]
    }
    \]
    is a Quillen equivalence if and only if $f$ is a Morita weak equivalence of algebraic theories.
\end{theo}
\begin{proof}
Suppose that the morphism of simplicial algebraic theories $f\colon \at{S}\to \at{T}$ is a Morita weak equivalence of algebraic theories, that is, the underlying map of simplicial categories is a Morita weak equivalence. We have to show that for $X$ cofibrant in $\ssets^{\at{S}}_{\wm}$ and
$Y$ fibrant in $\ssets^{\at{T}}_{\wm}$, a morphism $\morp{\alpha}{X}{\tilde{f^*}Y}$ is an $L_\at{S}$\nobreakdash-local equivalence if and only if the adjoint $\morp{\bar{\alpha}}{\tilde{f_!}X}{Y}$ is an $L_\at{T}$\nobreakdash-local equivalence.

First, we factor the map $\alpha$ into an $L_\at{S}$\nobreakdash-local trivial cofibration $\beta$ followed by a an $L_\at{S}$\nobreakdash-local fibration $\gamma$, and we consider the following commutative diagrams:
    \[
    \xymatrix{X \ar@{>->}[rd]_-{\beta}^{\sim} \ar[rr]^-{\alpha} & & \tilde{f^*}Y\\
        & X' \ar@{>>}[ru]_-{\gamma}&}
    \qquad
    \xymatrix{\tilde{f_!}X \ar[rd]_-{\tilde{f_!}\beta} \ar[rr]^-{\bar{\alpha}} & & Y\\
        & \tilde{f_!}X' \ar[ru]_-{\bar{\gamma}}.&
    }
    \]
 Thus, $\alpha$ is an $L_\at{S}$\nobreakdash-local equivalence if and only if $\gamma$ is a projective trivial fibration.
But  $\gamma$ is a projective weak equivalence if and only if $\bar{\gamma}$ is a projective weak equivalence.
    Since $\bar{\gamma}$ is a map between weak models, $\bar{\gamma}$ is a projective weak equivalence if and only if it is an $L_\at{T}$\nobreakdash-local equivalence by Proposition~\ref{prop:webetweenwm}. Finally, $\bar{\gamma}$ is a $L_\at{T}$\nobreakdash-local equivalence if and only if $\bar{\alpha}$ is.

    Conversely, suppose that $(\tilde{f_!},\tilde{f^*})$ is a Quillen equivalence. Then, by Lemma \ref{lemma:quequnit} the unit and the counit of $(\tilde{f_!},\tilde{f^*})$ are weak equivalences on weak models. It follows that $f$ is a Morita weak equivalence by Proposition \ref{prop:reqalg}, since corepresentable functors are weak models.
\end{proof}

\begin{coro}\label{cor:char_Morita_at}
A map $\morp{f}{\at{S}}{\at{T}}$ of simplicial algebraic theories is a Morita weak equivalence if and only if the induced Quillen pair
\[
    \xymatrix{
        f_!\colon\sAlg{\at{S}} \ar@<3pt>[r] &
        \sAlg{\at{T}}\colon f^*	\ar@<3pt>[l]
    }
    \]
is a Quillen equivalence.
\end{coro}
\begin{proof}
The result follows from Theorem~\ref{thm:char_Morita}, Remark~\ref{rem:adj_tilde_f} and Theorem~\ref{thm:ber_bad}.
\end{proof}
\section{Symmetric operads, cartesian operads and algebraic theories}\label{sect:cart_oper_and_alg_th}\label{sect:Oper_AT}

We have ended the previous section with the desired characterization of Morita equivalences between simplicial algebraic theories. In the upcoming sections we are going to recall the connection that binds algebraic theories and operads pointed out by Kelly in \cite{Ke05}, that will allow us to exploit the results obtained for algebraic theories to prove a similar characterization of Morita equivalences for (simplicial) coloured operads.

We start by recalling the definition of  symmetric operad and cartesian operad enriched in a cartesian category $\Mcc$. Cartesian operads are just a different presentation of algebraic theories (see Section~\ref{sec:opToAt}). The insight given by Kelly is that cartesian operads (algebraic theories) and symmetric operads are basically the same construct in two different contexts: the cartesian one and the symmetric monoidal one, respectively.

Although classically coloured operads are defined as sequences of objects together with a composition product and a unit (see, for instance, \cite{BM07} for an account), for the purpose of this paper it will be more useful to define operads as monoids in certain categories of collections.

Even though the definitions and results of this section will be given for a general cartesian category $\Mcc$, we will only need to consider the cases in which $\Mcc$ is $\sets$ or $\ssets$ in the rest of the paper.

\subsection{Symmetric and cartesian ordered sequences}
Let $\Fin$ be a skeleton for the category of finite sets. Each object of $\Fin$ is uniquely determined by its cardinality. Let $\Sym$ be the wide subcategory of $\Fin$ spanned by all the isomorphisms.

Let $C$ be a fixed set. We will denote by $\Fin_C$ and $\Sym_C$ the comma categories $\Fin \downarrow C$ and $\Sym \downarrow C$, respectively. Their objects can be represented as finite sequences $\str{c}=(c_1,\ldots, c_n)$ in $C$. For such a sequence $\str{c}$ we denote by $|\str{c}|$ its cardinality. The sequence of cardinality zero will be denoted by $[-]$. The categories $\Fin_C\op$ and $\Sym_C\op$ can be characterized as being the free cartesian category generated by $C$ and the free symmetric monoidal category generated by $C$, respectively.

Let $\Sign{C}$ and $\Fign{C}$ denote the categories $\Sym_C\times C$ and $\Fin_C\times C$, respectively.
The objects of these categories will be written as $(\str{c}; c)$, where $\str{c}$ is an object of $\Sym_C$ or
$\Fin_C$, respectively, and $c\in C$. The inclusion of $\Sign{C}$ into $\Fign{C}$ will be denoted by
\begin{equation}\label{eq:inclusion_symfin}
    \lmorp{p}{\Sign{C}}{\Fign{C}}.
\end{equation}
Note that both $\Fin_C$ and $\Sym_C$ have as set of objects the set of finite ordered sequences of elements of $C$.

Let $(\Mcc,\times,\ast)$ be a bicomplete closed cartesian category. We will call $\Mcc^{\Sign{C}}$ and  $\Mcc^{\Fign{C}}$ the category of \emph{symmetric $C$-coloured $\Mcc$-collections} and the category of \emph{cartesian $C$-coloured collections}, respectively.

We will denote by $\Dayp$ both Day convolution tensor products on the categories $\Mcc^{\Fin_C}$ and $\Mcc^{\Sym_C}$. Observe that $\Dayp$ on $\Mcc^{\Fin_C}$ is isomorphic the cartesian product \cite[\S 8]{Ke05}.

Every cartesian $C$-collection $O$ in $\Mcc^{\Fign{C}}$ defines a functor
\[
\gfun{O(-;-)}{C}{\Mcc^{\Fin_C}}{c}{O(-;c).}
\]
By the universal property of $\Fin_C\op$, this extends to a product preserving functor
\[
\gfun{O^{(-)}}{\Fin_C\op}{\Mcc^{\Fin_C}}{(c_1,\dots,c_n)}{O(-;c_1) \Dayp O(-;c_2) \Dayp \dots \Dayp O(-;c_n).}
\]
In the same way, for every symmetric $C$-collection $O$ in $\Mcc^{\Sign{C}}$ we get a functor
\[
\gfun{O^{(-)}}{\Sym_C\op}{\Mcc^{\Sym_C}}{(c_1,\dots,c_n)}{O(-;c_1) \Dayp O(-;c_2) \Dayp \dots \Dayp O(-;c_n).}
\]
\subsection{Symmetric operads and cartesian operads}
There is a non\nobreakdash-symmetric monoidal product $\odot$ on $\Mcc^{\Sign{C}}$ defined as the coend:
\[
O\odot P \cong \int^{\str{c}\in \Sym_C} O({\str{c}}) \times P^{\str{c}}.
\]
In this formula, we identify $O$ with an object of $(\Mcc^{\Fin_C})^C$ and $P^{(-)}$ as an object of $(\Mcc^{\Fin_C\op})^{\Fin_C}$ so that the (parametrized) coend defined on the right is indeed an object of $\Mcc^{\Fign{C}}$.
The unit for this monoidal product is given by the $C$\nobreakdash-collection~$I_C$ given by:
\[
I_C(\str{c};c) = \begin{cases}
\ast & \text{if } \str{c}=c,\\
\emptyset & \text{otherwise.}
\end{cases}
\]

Similarly, there is a non\nobreakdash-symmetric monoidal product $\odot$ on $\Mcc^{\Fign{C}}$ defined as the coend:
\[
O\odot P \cong \int^{\str{c}\in \Fin_C} O({\str{c}}) \times P^{\str{c}}.
\]
The unit of this product is the object $I'_C$ of $\Mcc^{\Fign{C}}$ defined as $I'_C(\str{c}; c)=\Fin_C(c, \str{c})$.

The category of monoids in $(\Mcc^{\Fin_C\op})^{\Fin_C}$ with the monoidal product $\odot$ is the category of \emph{$C$-coloured symmetric operads in $\Mcc$}, denoted by $\EOperc{\Mcc}{C}$. The category of \emph{$C$-coloured cartesian operads in $\Mcc$} (or \emph{clones}), denoted by ${\ECloc{\Mcc}{C}}$, is the category of monoids in $(\Mcc^{\Fign{C}},\odot,I_C')$.

Given a $C$-coloured operad $\mathcal{O}$ in $\EOperc{\Mcc}{C}$, the category of \emph{$\mathcal{O}$-algebras in $\Mcc$}, denoted by $\Alg({\mathcal{O}})$, is the category of left $\mathcal{O}$-modules concentrated in arity 0. In other words, the functor $\zeta\colon C\to \Sign{C}$ that sends $c$ to
$([-];c)$ induces a functor $\zeta_!\colon \Mcc^C\to \Mcc^{\Sign{C}}$ by left Kan extension. An $\mathcal{O}$-algebra is then an object $X$ of $\Mcc^C$ together with a left action $\mathcal{O}$-action $\mathcal{O}\odot \zeta_!(X)\to \zeta_!(X)$.
Algebras over cartesian coloured operads are defined similarly.

Consider the left extension-restriction adjunction between symmetric and cartesian collections, induced by the morphism $p$ defined in (\ref{eq:inclusion_symfin})
\[
\xymatrix{
    {p_!}\colon {\Mcc^{\Sign{C}}}  \ar@<3pt>[r] &
    {\Mcc^{\Fign{C}}}	\ar@<3pt>[l] \colon {p^*}.
}
\]
\begin{propo}
The adjunction $(p_!,p^*)$ restricts to an adjunction between $C$\nobreakdash-co\-lou\-red symmetric operads and $C$-coloured cartesian operads in $\Mcc$:
    \[
    \xymatrix{
      p_!\colon  {\EOperc{\Mcc}{C}}  \ar@<3pt>[r] &
        {\ECloc{\Mcc}{C}}	\ar@<3pt>[l]\colon p^*.
    }
    \]
\end{propo}
\begin{proof}
The functor $p_!$ preserves the monoidal product \cite[\S 8]{Ke05} and $I'_C=p_!(I_C)$. Hence, $p_!$ is strong monoidal functor.  It follows that the right adjoint $p^*$ is a lax monoidal functor. So both functors preserve monoids.
\end{proof}
In the follow up we will need a more explicit description of the left adjoint $p_!$. For every $\str{c}$ in $\Fign{C}$ let us denote by $\Ord{\str{c}}$ the full subcategory of $p\downarrow \str{c}$ spanned by those objects $\morp{f}{\str{b}}{\str{c}}$ such that the underlying map of finite sets
\[\morp{f}{\{1,\dots,|\str{b}|\}}{\{1,\dots,|\str{c}|\}}
\] is an ordered map.
\begin{lemma}
For every $\str{c}$ in $\Fign{C}$ the inclusion $i\colon \Ord{\str{c}}\to p\downarrow \str{c}$ is final.
\end{lemma}
\begin{proof}
We have to show that for every object $d$ in $p\downarrow \str{c}$, the comma category $d\downarrow i$ is non-empty and connected. This follows from the fact that every map of finite sets factors into a bijection followed by an ordered map, and the ordered map is uniquely determined.
\end{proof}
The preceding lemma can be used to simplify the computation of $p_!$.
Explicitly, for every operad $\ope{O}$ in $\EOperc{\Mcc}{C}$ and every $(\str{c};d)$ in $\Fign{C}$ we have that
\[
\begin{split}
(p_!\ope{O})(\str{c};d)&\cong \int^{\str{b}\in \Sym_C} \Fin_C(\str{b},\str{c})\times \ope{O}(\str{b};d)\cong\\
\underset{\{(\str{b},d) \to (\str{c},d)\}\in p\downarrow\str{c}}\colim \ope{O}(\str{b};d) &
\cong \underset{\{f:(\str{b},d)\to(\str{c},d)\}\in \Ord{\str{c}}}{\coprod} \ope{O}(\str{b};d)/\Sigma_f,
\end{split}
\]
where, for every $f\colon(\str{b},d)\to(\str{c},d)$ in $\Ord{\str{c}}$, the group $\Sigma_f$ is defined to be the subgroup of $\Sym_C(\str{b},\str{b})$ spanned by the automorphisms that fix the fibers of $f$, this is,
\[
\Sigma_f = \{ \sigma \in \Sym_C(\str{b},\str{b}) \mid f\sigma(i) = f(i) \text{ for every }i\in \{1,\dots,|\str{b}| \}\}.
 \]
Here we identify $\Sym_C(\str{b},\str{b})$ with a subset of $\Sigma_{|\str{b}|}$.

\subsection{From symmetric operads to algebraic theories}\label{sec:opToAt}

We can associate to every cartesian $C$-coloured $\Mcc$-operad $\mathcal{O}$ a $C$-sorted $\Mcc$-algebraic theory $\cAt{\mathcal{O}}$ as follows: for every $\str{c},\str{d}\in \Str{C}$ we set
\[
\cAt{\mathcal{O}}(\str{c},\str{d}) \cong \mathcal{O}^{\str{d}}(\str{c}) \cong \prod_{i=1}^{|\str{d}|}\mathcal{O}(\str{c};d_i).
\]
One can check that the category of $\cAt{\mathcal{O}}$-algebras is equivalent to the category of $\mathcal{O}$-algebras in $\Mcc$.

This construction defines a functor $\cAt{}\colon \ECloc{\Mcc}{C}\to \EAtc{\Mcc}{C}$ which is actually an equivalence of categories. Its inverse functor $\mathbb{C}$ assigns to every $C$\nobreakdash-sorted $\Mcc$\nobreakdash-algebraic theory $\cat{T}$ the cartesian $\Mcc$-operad $\mathbb{C}(\cat{T})$ defined as
$$\mathbb{C}(\cat{T})(\str{c};d) \cong \cat{T}(\str{c},d), \mbox{ for every  }(\str{c};d)\mbox{ in }\Fign{C}.
$$

We can now compose the functors $p_!$ and $\cAtf$ to get a functor
\[
\lmorp{\oatf}{\EOperc{\Mcc}{C}}{\EAtc{\Mcc}{C}}
\]
that associates to each symmetric $\Mcc$-operad an $\Mcc$-algebraic theory.
Explicitly, for every $\ope{O}$ in $\EOperc{\Mcc}{C}$ and every $\str{c},\str{d}\in \Str{C}$, we have that
\begin{equation}\label{mor_TO}
    \oat{\ope{O}}(\str{c},\str{d})\cong \prod_{i=1}^{|\str{d}|} \left(\underset{\{f:(\str{a};d_i)\to(\str{c};d_i)\}\in \Ord{(\str{c}; d_i)}}{\coprod} \ope{O}(\str{a};d_i)/\Sigma_f\right).
\end{equation}
Thus, a morphism $\mathbf{f}$ in $\oat{\ope{O}}(\str{c},\str{d})$ will be a sequence $\mathbf{f}=\{[f_i,\str{a}_i,o_i]\}_{i=1}^{|\str{d}|}$, where $f_i\colon (\str{a}_i;d_i)\to (\str{c};d_i)$ in $\Ord{(\str{c};d_i)}$ and $o_i\in\ope{O}(\str{a}_i; d_i)/\Sigma_{f_i}$. Composition in $\oat{\ope{O}}$ is defined as follows: if  $\mathbf{f}=\{[f_i,\str{a}_i,o_i]\}_{i=1}^{|\str{d}|}$ in $\oat{\ope{O}}(\str{c},\str{d})$ and $\mathbf{g}=\{[g_j,\str{b}_j,q_j]\}_{j=1}^{|\str{e}|}$ in $\oat{\ope{O}}(\str{d},\str{e})$, then

\begin{multline}
\mathbf{gf}=\{[f_{g_1(1)}\sqcup \dots \sqcup f_{g_1(|\str{b}_1|)} \sqcup \dots \sqcup f_{g_k(1)}\sqcup \dots \sqcup f_{g_k(|\str{b}_k|)},\\ \notag \str{a}_{g_1(1)}\ldots\str{a}_{g_k(|\str{b}_k|)}, (o_{g_1(1)},\ldots, o_{g_k(|\str{b}_k|)})\circ q_k]\}_{k=1}^{|\str{e}|}
\end{multline}
    
The category of $\ope{O}$-algebras is equivalent to the category of $\oat{\ope{O}}$-algebras in $\Mcc$.

\subsection{The category of coloured operads}\label{sect:ch_colour}
Let $\morp{f}{C}{D}$ be a function between sets. This function extends to two functors
\[
\gfun{f}{\Sign{C}}{\Sign{D}}{(c_1,\dots,c_n)}{(f(c_1),\dots,f(c_n)),}\
\gfun{f}{\Fign{C}}{\Fign{D}}{(c_1,\dots,c_n)}{(f(c_1),\dots,f(c_n)).}
\]
The restriction functors
\[
\begin{split}
\morp{f^*}{\Mcc^{\Sign{D}}}{\Mcc^{\Sign{C}}}\;\;{\rm and}\;\;
\morp{f^*}{\Mcc^{\Fign{D}}}{\Mcc^{\Fign{C}}}
\end{split}
\]
are both monoidal, therefore they restrict to functors
\begin{gather*}
    \morp{f^*}{\EOperc{\Mcc}{D}}{\EOperc{\Mcc}{C}}\;\;{\rm and}\;\;
    \morp{f^*}{\ECloc{\Mcc}{D}}{\ECloc{\Mcc}{C}}.
\end{gather*}
There is also an evident functor
\[
\morp{f^*}{\EAtc{\Mcc}{D}}{\EAtc{\Mcc}{C}}.
\]
It is easy to see that these assignments are natural in $f$. In other words they extend to functors:
\begin{gather*}
    \gfun{\mathsf{Op}}{\sets}{\Cat}{C}{\EOperc{\Mcc}{C},}\!\!\!\!
     \gfun{\mathsf{COp}}{\sets}{\Cat}{C}{\ECloc{\Mcc}{C},}\!\!\!\!
    \gfun{\mathsf{ATh}}{\sets}{\Cat}{C}{\EAtc{\Mcc}{C}.}
\end{gather*}
We can apply the (contravariant) Grothendieck construction to $\mathsf{Op}$ and $\mathsf{COp}$ to get two categories $\EOper{\Mcc}$ and $\EClo{\Mcc}$ fibered over $\sets$.
The category $\EOper{\Mcc}$ is called the category of \emph{symmetric coloured $\Mcc$-operads} while the category $\EClo{\Mcc}$ will be called the category of \emph{cartesian coloured $\Mcc$-operads}.

In the same way we can construct the category (fibered over $\sets$) of multi-sorted algebraic theories $\EAt{\Mcc}$. The category $\EAt{\Mcc}$ is equivalent to $\EClo{\Mcc}$.
The functors $p_!$ and $\cAtf$ considered in the previous sections define (pseudo)natural transformations between $\mathsf{Op}$ and $\mathsf{COp}$, and $\mathsf{COp}$ and $\mathsf{ATh}$, respectively. In fact, $p_! f^*\cong f^* p_!$ and $\cAtf f^*\cong f^*\cAtf$, for every function $\morp{f}{C}{D}$.

Via the Grothendieck construction, these natural transformations correspond to a diagram of functors
\begin{equation}\label{diag:Groth_adj}
\xymatrix{
    & \EClo{\Mcc} \ar[rd]^-{\cAtf} & \\
    \EOper{\Mcc} \ar[ru]^-{p_!} \ar[rr]_-{\oatf}& &
    \EAt{\Mcc},
}
\end{equation}
which is natural in $\Mcc$ for functors that preserve products and colimits.

\section{Morita equivalences of coloured operads}\label{sec:operads}
Recall that for every cocomplete cartesian category $\Mcc$ there is an adjunction between the category of small $\Mcc$-enriched categories and the category of $\Mcc$-enriched coloured operads
\begin{equation}\label{eq:adj_cat_oper}
j_!:\ECat{\Mcc}\myrightleftarrows{\rule{0.4cm}{0cm}} \EOper{\Mcc} :j^*.
\end{equation}
The right adjoint $j^*$ associates to every $C$\nobreakdash-coloured operad its underlying category with set of objects $C$. More explicitly, $j^*(\ope{O})(c,d)\cong \ope{O}(c;d)$ for every $c, d \in C$ and composition and identities are inherited directly from the ones in $\ope{O}$. Note that a morphism of operads $f$ is essentially surjective if and only if $j^*(f)$ is essentially surjective.

\subsection{A characterization of Morita equivalences of coloured operads}
In this section we characterize the Morita equivalences between operads in $\sets$ as those morphisms of coloured operads which induce an equivalence of categories between the respective categories of algebras. We will use the notation and results of Section~\ref{sect:cart_oper_and_alg_th} in the particular case $\mathcal{M}=\sets$.

\begin{defi}
    A morphism of coloured operads $f$ is a \emph{Morita equivalence} if it is fully faithful and the functor $j^*(f)$ is essentially surjective up to retracts.
\end{defi}

Our goal is to prove that a morphism $f$ of coloured operads in $\sets$ is a Morita equivalence if and only if the morphism of algebraic theories $\oat{f}$ is a Morita equivalence of algebraic theories, or equivalently, if and only if the induced adjunction $(f_!,f^*)$ between the categories of algebras is an equivalence of categories.

\begin{lemma}\label{lem:fully_faith}
    A morphism $f$ between symmetric coloured operads in $\sets$ is fully faithful if and only if the associated functor of algebraic theories $\oat{f}$ is fully faithful.
\end{lemma}
\begin{proof}
    Let $f\colon\ope{O}\to \ope{P}$ be a morphisms of symmetric coloured operads.  Let $C$ be the set of colours of $\ope{P}$.  For every $(\str{c};c_0)$ in $\Sign{C}$ the following diagram commutes:
    \begin{equation}\label{diag:ff}
    \xymatrix@C1.7cm{\ope{O}(\str{c};c_0)
        \ar[r]^-{f_{(\str{c};c_0)}}
        \ar[d]_-{\langle \id_{\str{c}}\rangle}
        &
        \ope{P}(f(\str{c});f(c_0))
        \ar[d]^-{\langle \id_{f(\str{c})}\rangle}
        \\
        \underset{\alpha:\{\str{b}\to \str{c}\} \in \Ord{\str{c}}}{\coprod} \ope{O}(\str{b};c_o)/\Sigma_\alpha
        \ar[r]^-{\underset{\alpha}\coprod f_{(\str{b};c_0)}/\Sigma_{\alpha}}
        &
        \underset{\alpha\in \Ord{\str{o}}}\coprod \ope{P}(f(\str{b});f(c_0))/\Sigma_\alpha,\\
    }
    \end{equation}
    where the bottom horizontal arrow is isomorphic to
    $$
    \xymatrix{\oat{\ope{O}}(\str{c},c_0)
        \ar[rr]^-{\oat{f}_{(\str{c};c_0)}}
        & &
        \oat{\ope{P}}(f(\str{c}),f(c_0)). 		
    }
    $$
    From the description of the components of $\oat{f}$ it is clear that if $f$ is fully faithful, then $\oat{f}$ is fully faithful.
    
    On the other hand, it follows from the commutativity of the above diagram that if $\oat{f}_{(\str{c},c_0)}$ is an isomorphism, then $f_{(\str{c};c_0)}$ is an isomorphism. Since $(\str{c},c_0)$ was chosen arbitrarily, this proves that if $\oat{f}$ is fully faithful then $f$ is fully faithful.
\end{proof}

\begin{lemma}
    Suppose that $\ope{O}$ is a symmetric $C$-coloured operad and let $c\in C$. If~$c$ is the retract of some $\str{d}=(d_1,\dots,d_n)$ in $\oat{\ope{O}}$, then there exists $1\leq j \leq n$ such that $c$ is a retract of $d_j$.
\end{lemma}
\begin{proof}
    Recall from \eqref{mor_TO} the notation for the morphisms and composition in $\oat{\ope{O}}$. By assumption, there exist
    \[\mathbf{f}=\{[f_i,\str{a}_i,o_i]\}_{i=1}^{|\str{d}|} \in \oat{\ope{O}}(c, \str{d})\;\;\mbox{and}
    \;\;\mathbf{g}= [g,\str{b},q] \in \oat{\ope{O}}(\str{d},c)
    \]
    such that $\mathbf{g}\mathbf{f}=\id_{c}$, that is,
    \[
    [f_{g(1)}\sqcup \dots \sqcup f_{g(|\str{b}|)}, \str{a}_{g(1)}\ldots \str{a}_{g(|\str{b}|)},(o_{g(1)},\dots,o_{g(|\str{b}|)})\circ q] = [\id_c, c,\id_c].
    \]
    This implies that there exists $1\leq k \leq |\str{b}|$ such that $\str{a}_{g(k)}=c$ and $\str{a}_{g(i)}=[-]$ for all $i\neq k$.
    Therefore $o_{g(k)}\in \ope{O}(c;d_{g(k)})$, $f_{g(k)}=\id_c$ and $o_i\in \ope{O}(-; d_i)$ for every $i\neq k$.
    
    Let $q'\in \ope{O}(d_{g(k)}; c)$ be the composition
    \[
    (o_{g(1)},\dots,o_{g(k-1)},\id_{d_{g(k)}},o_{g(k+1)},\dots,o_{g(|\str{b}|)})\circ q
    \]
    and set
    \[
    \mathbf{f}'=[f_{g(k)},c, o_{g(k)}] \in \oat{\ope{O}}(c,d_{g(k)})\;\;\mbox{and}\;\;
    \mathbf{g}'=[\id_{d_{g(k)}},d_{g(k)},q'] \in \oat{\ope{O}}(d_{g(k)},c).
    \]
    By construction $\mathbf{g}'\mathbf{f}'=\id_c$ and therefore $c$ is a retract of $d_{g(k)}$.
\end{proof}

\begin{coro}\label{coro:key_retract}
    A morphism $f$ of symmetric coloured operads in $\sets$ is essentially surjective up to retracts if and only if $\oat{f}$ is essentially surjective up to retracts.
\end{coro}

\begin{theo}\label{theo:moritaOpSet}
    Let $\morp{f}{\ope{O}}{\ope{P}}$ be a morphism of symmetric coloured operads in $\sets$. The following are equivalent:
    \begin{itemize}
        \item[{\rm (i)}] The morphism $f$ is a Morita equivalence of operads.
        \item[{\rm (ii)}] The functor $\oat{f}$ is a Morita equivalence of algebraic theories.
        \item[{\rm (iii)}] The induced adjunction
        \[
        f_!:\Alg(\ope{O})\myrightleftarrows{\rule{0.4cm}{0cm}} \Alg(\ope{P}): f^*
        \]
        is an equivalence of categories.
    \end{itemize}
\end{theo}
\begin{proof}
    By Lemma~\ref{lem:fully_faith} and Corollary~\ref{coro:key_retract} the morphism $f$ is fully faithful and essentially surjective up to retracts if and only if $\oat{f}$ is so.  This proves the equivalence between (i) and (ii).  The equivalence between (ii) and (iii) is immediate since the categories of algebras for $\ope{O}$ and $\oat{O}$ are equivalent.
\end{proof}
\subsection{A characterization of Morita equivalences of simplicial operads}

Let $\mathcal{O}$ be a $C$-coloured operad in simplicial sets. Its category of algebras $\Alg({\mathcal{O}})$ admits a transferred model structure via the free-forgetful adjunction, which is called the \emph{projective model structure}; see \cite[Theorem 2.1]{BM07}. The weak equivalences and the fibrations are the morphisms $X\to Y$ such that $X(\mathbf{c}; d)\to Y(\mathbf{c};d)$ is a weak equivalence or a fibration of simplicial sets for every $(\mathbf{c}, d) \in \Sign{C}$, respectively. Moreover, for a morphism of operads $f\colon \mathcal{O}\to\mathcal{P}$ the induced adjunction
\[
f_!:\Alg(\ope{O})\myrightleftarrows{\rule{0.4cm}{0cm}} \Alg(\ope{P}): f^*
\]
is a Quillen pair with respect to the corresponding projective model structures.

Recall that a symmetric $C$-coloured operad $\ope{O}$ is called \emph{$\Sigma$-cofibrant} if for every~$(\str{c}, c)$ in $\Sign{C}$ the $\Sym_C(\str{c},\str{c})$-module $\ope{O}(\str{c};c)$ is cofibrant in $\ssets^{\Sym_C(\str{c},\str{c})}$ endowed with the projective model structure.

The following definition extends the definition of Morita weak equivalence between simplicial categories to simplicial operads (see \cite[Definition 5.3]{CG19}):
\begin{defi}
  Let $f\colon \ope{O} \to \ope{P}$ be a morphism of simplicial coloured operads. We say that
    \begin{itemize}
        \item[{\rm (i)}] $f$ is \emph{homotopically fully faithful} if $f\colon \ope{O}(\mathbf{c}, c') \to \ope{P}(f(\mathbf{c}), f(c'))$ is a weak equivalence for every signature $(\mathbf{c}, c')\in \Sign{C}$, where $C$ is the set of colours of $\ope{O}$;
        \item[{\rm (ii)}] $f$ is \emph{homotopically essentially surjective up to retracts} if the functor $\pi_0(j^*(f))$ is essentially surjective up to retracts;
        \item[{\rm (iii)}] $f$ is a \emph{Morita weak equivalence} if it is homotopically fully faithful and homotopically essentially surjective up to retracts.
    \end{itemize} 
\end{defi}

\begin{propo}\label{propo:hom fully faith}
    A map $f$ between $\Sigma$-cofibrant simplicial coloured operads is homotopically fully faithful if and only if the associated map $\oat{f}$ of simplicial algebraic theories is homotopically fully faithful.
\end{propo}
\begin{proof}
    The proof uses the same arguments as in the proof of Lemma~\ref{lem:fully_faith}. The main difference being
    that in this case we are dealing with weak equivalences instead of isomorphisms.
    
    We need to check that in the commutative diagram~(\ref{diag:ff}) the top horizontal map is a weak equivalence if and only if the bottom horizontal map is a weak equivalence. This will follow at once from the following facts about projective model structures on simplicial sets (we use the same notation as in the proof of Lemma~\ref{lem:fully_faith}):
    \begin{itemize}
        \item[{\rm (i)}] Consider the inclusion of groups $\Sigma_{\alpha}\to \Sym_C(\str{c},\str{c})$. Then the forgetful functor
        \[
        \ssets^{\Sym_C(\str{c},\str{c})} \longrightarrow \ssets^{\Sigma_{\alpha}}
        \]
        between the corresponding projective model structures preserves cofibrant objects; see \cite[Lemma 2.5.1]{BM06}.
        \item[{\rm (ii)}] If $\morp{f}{X}{Y}$ is a $\Sigma_{\alpha}$-equivariant map between cofibrant objects in $\ssets^{\Sigma_{\alpha}}$, which is a projective weak equivalence, then $\morp{f/\Sigma_{\alpha}}{X/\Sigma_{\alpha}}{Y/\Sigma_{\alpha}}$ is a weak equivalence of simplicial sets.
        \item[{\rm (iii)}] A coproduct of maps $\coprod_{i\in I} f_i$ in $\ssets$ is a weak equivalence if and only if $f_i$ is a weak equivalence for every $i\in I$. \qedhere
        
    \end{itemize}
\end{proof}

Consider the functor $\pi_0\colon \ssets\to \sets$.  This functor preserves products and colimits and thus, by naturality of diagram~(\ref{diag:Groth_adj}), it induces a commutative diagram
\[
\xymatrix{
    \sOper \ar[r]^-{\oatf} \ar[d]_-{\pi_0} &
    \sAt \ar[d]^-{\pi_0} \\
    \Oper \ar[r]_-{\oatf} &
    \At. \\
}
\]

\begin{propo}\label{prop:hom ess retract}
    A map $f$ between simplicial coloured operads is homotopically essentially surjective up to retracts if and only if $\oat{f}$ is homotopically essentially surjective up to retracts.
\end{propo}
\begin{proof}
    A map of simplicial operads is homotopically essentially surjective up to retracts if the map of simplicial categories $\pi_0(f)$ is essentially surjective up to retracts. By Corollary~\ref{coro:key_retract}, $\pi_0(f)$ is essentially surjective up to retracts if and only if $\oat{\pi_0(f)}\cong\pi_0(\oat{f})$ is essentially surjective up to retracts. But this happens if and only if $\oat{f}$ is homotopically essentially surjective up to retracts.
\end{proof}
\begin{theo}\label{theo:main morita oper}
    Let  $\morp{f}{\ope{O}}{\ope{P}}$ be a map between $\Sigma$-cofibrant simplicial coloured operads. The following are equivalent:
    \begin{enumerate}
        \item[{\rm (i)}] The map $f$ is a Morita weak equivalence of simplicial operads.
        \item[{\rm (ii)}] The functor $\oat{f}$ is a Morita weak equivalence of simplicial algebraic theories.
        \item[{\rm (iii)}] The Quillen pair
        \[
        f_!:\Alg(\ope{O})\myrightleftarrows{\rule{0.4cm}{0cm}} \Alg(\ope{P}): f^*
        \]
        is a Quillen equivalence.
    \end{enumerate}
\end{theo}
\begin{proof}
    The equivalence between (i) and (ii) follows directly from Proposition~\ref{propo:hom fully faith} and Proposition~\ref{prop:hom ess retract}. The equivalence between (ii) and (iii) follows from Corollary~\ref{cor:char_Morita_at} and the fact that for every operad $\ope{O}$, the categories of $\ope{O}$-algebras and $\oat{\ope{O}}$-algebras are equivalent.
\end{proof}

\begin{rem}
The fact that (i) implies (iii) was proved in \cite[Theorem 4.1]{BM07} for operads with a fixed set of colours $C$ in a monoidal model category. Our result, which is an if and only if, works for the category of all coloured operads in simplicial sets and hence implies the one by Berger and Moerdijk in the simplicial case. However, our proof is different from the one in~\cite{BM07}, since we use simplicial algebraic theories, the theory of weak models developed by Badzioch and Bergner, and a functor from simplicial operads to simplicial algebraic theories.
\end{rem}
\begin{rem}
Note that the equivalence of implications (ii) and (iii) can be proved without the $\Sigma$-cofibrancy conditions. At first, the fact that among the class of maps that we have defined as Morita weak equivalences, only those between $\Sigma$-cofibrant operads induce homotopy equivalences on the homotopy categories of algebras might seems disappointing. However, one has to keep in mind that this was already the case for the class of weak equivalences in the Cisinski--Moerdijk model structure \cite{CM13} on $\sOper$, which models the homotopy theory of $\infty$-operads. We also recall that every cofibrant operad in the Cisinski--Moerdijk model structure is $\Sigma$\nobreakdash-cofibrant and that the class of Morita weak equivalence includes the Cisinski-Moerdijk weak equivalences. In \cite[Theorem 5.7]{CG19} the authors have proved that there exists a localization of the Cisinski--Moerdijk model structure, called the Morita model structure, that has the Morita weak equivalences as class of weak equivalences.
    
The above theorem shows that Morita weak equivalences are the only morphisms of $\infty$-operads inducing an equivalence between the homotopy invariant algebraic structures determined by source and target. The Morita model structure can therefore be regarded as the homotopy theory of homotopy invariant algebraic structures.
\end{rem}


\begin{thebibliography}{BdBM17}
    
    \bibitem[Bad02]{BA02}
    B.~Badzioch.
    \newblock Algebraic theories in homotopy theory.
    \newblock {\em Ann. of Math. (2)}, 155(3):895--913, 2002.

    \bibitem[Ber06]{Ber06}
    J.~Bergner.
    \newblock Rigidification of algebras over multi-sorted theories.
    \newblock {\em Algebr. Geom. Topol.}, 6:1925--1955, 2006.

    \bibitem[BM06]{BM06}
    C.~Berger and I.~Moerdijk.
    \newblock The {B}oardman--{V}ogt resolution of operads in monoidal model
    categories.
    \newblock {\em Topology}, 45(5):807--849, 2006.
    
    \bibitem[BM07]{BM07}
    C.~Berger and I.~Moerdijk.
    \newblock Resolution of coloured operads and rectification of homotopy
    algebras.
    \newblock In {\em Categories in algebra, geometry and mathematical physics},
    volume 431 of {\em Contemp. Math.}, pages 31--58. Amer. Math. Soc.,
    Providence, RI, 2007.

\bibitem[BD86]{BD86}  
F.~Borceux and D.~Dejean.
\newblock Cauchy completion in category theory.
\newblock {\em Cahiers Topologie G\'eom. Diff\'erentielle Cat\'eg.},
  27(2):133--146, 1986.

    \bibitem[CG19]{CG19}
    G.~Caviglia and J.~J. Guti\'errez.
    \newblock Morita homotopy theory for $(\infty,1)$-categories and $\infty$-operads.
    \newblock {\em Forum Math.} 31(3): 661--684, 2019.

    \bibitem[CM13]{CM13}
    D.-C. Cisinski and I.~Moerdijk.
    \newblock Dendroidal sets and simplicial operads.
    \newblock {\em J. Topol.}, 6(3):705--756, 2013.

    \bibitem[DK87]{DK87}
    W.~G. Dwyer and D.~M. Kan.
    \newblock Equivalences between homotopy theories of diagrams.
    \newblock In {\em Algebraic topology and algebraic {$K$}-theory ({P}rinceton,
        {N}.{J}., 1983)}, volume 113 of {\em Ann. of Math. Stud.}, pages 180--205.
    Princeton Univ. Press, Princeton, NJ, 1987.
   
    \bibitem[EZ76]{EZ76}
    B.~Elkins and J.~A. Zilber.
    \newblock Categories of actions and {M}orita equivalence.
    \newblock {\em Rocky Mountain J. Math.}, 6(2):199--225, 1976.
    
    \bibitem[GJ99]{GJ99}
    P.~G. Goerss and J.~F. Jardine.
    \newblock {\em Simplicial homotopy theory}, volume 174 of {\em Progress in
        Mathematics}.
    \newblock Birkh\"auser Verlag, Basel, 1999.

    \bibitem[Kel05]{Ke05}
    G~M. Kelly.
    \newblock On the operads of {J}.{P}. {M}ay.
    \newblock {\em Repr. Theory Appl. Categ}, 13:1--13, 2005.

    \bibitem[May67]{May67}
    J.~P. May.
    \newblock {\em Simplicial objects in algebraic topology}.
    \newblock Van Nostrand Mathematical Studies, No. 11. D. Van Nostrand Co., Inc.,
    Princeton, N.J.-Toronto, Ont.-London, 1967.
   
    \bibitem[Rez02]{Re02}
    C.~Rezk.
    \newblock Every homotopy theory of simplicial algebras admits a proper model.
    \newblock {\em Topology Appl.}, 119(1):65--94, 2002. 
\end{thebibliography}
\end{document}